\documentclass[a4,12pt]{amsart}
\usepackage{amssymb,amstext,amsmath,amscd,amsthm,amsfonts,enumerate,latexsym}
\usepackage{color}
\usepackage[dvipdfmx]{graphicx}
\usepackage[all]{xy}

\usepackage[dvipdfmx]{hyperref}

\theoremstyle{plain}
\newtheorem{thm}{Theorem}[section]

\newtheorem*{thm*}{Theorem}
\newtheorem*{cor*}{Corollary}

\newtheorem{prop}[thm]{Proposition}

\newtheorem{lem}[thm]{Lemma}
\newtheorem{cor}[thm]{Corollary}
\newtheorem{corollary}[thm]{Corollary}

\newtheorem*{claim*}{Claim}

\theoremstyle{definition}

\newtheorem{ex}[thm]{Example}
\newtheorem{Example}[thm]{Example}
\newtheorem{remark}[thm]{Remark}

\theoremstyle{remark}

\numberwithin{equation}{thm}

\newtheorem*{ac}{Acknowledgments}

\def\Min{\operatorname{Min}}

\def\Im{\operatorname{Im}}

\def\Ker{\operatorname{Ker}}

\def\Max{\operatorname{Max}}

\def\m{\mathfrak m}
\def\n{\mathfrak n}

\def\p{\mathfrak p}

\newcommand{\calF}{\mathcal{F}}

\newcommand{\calR}{\mathcal{R}}

\newcommand{\calT}{\mathcal{T}}

\newcommand{\fkp}{\mathfrak{p}}

\newcommand{\mapright}[1]{%
\smash{\mathop{%
\hbox to 1cm{\rightarrowfill}}\limits^{#1}}}

\newcommand{\mapleft}[1]{%
\smash{\mathop{%
\hbox to 1cm{\leftarrowfill}}\limits_{#1}}}

\def\depth{\operatorname{depth}}

\def\height{\mathrm{ht}}
\def\Spec{\operatorname{Spec}}

\title[Construction of strictly closed rings]{Construction of strictly closed rings}

\author[Naoki Endo]{Naoki Endo}
\address{Department of Mathematics, Purdue University, West Lafayette, IN 47907 U.S.A}
\email{nendo@purdue.edu}
\urladdr{https://www.math.purdue.edu/~nendo/}

\author[Shiro Goto]{Shiro Goto}
\address{Department of Mathematics, School of Science and Technology, Meiji University, 1-1-1 Higashi-mita, Tama-ku, Kawasaki 214-8571, Japan}
\email{shirogoto@gmail.com}

\thanks{2020 {\em Mathematics Subject Classification.} 13A15, 13B22, 13B30.}
\thanks{{\em Key words and phrases.} Strictly closed ring, Arf ring, weakly Arf ring, Stanley-Reisner ring}

\thanks{The first author was partially supported by JSPS Grant-in-Aid for Young Scientists 20K14299 and JSPS Overseas Research Fellowships. The second author was partially supported by JSPS Grant-in-Aid for Scientific Research (C) 16K05112.}

\begin{document}

\maketitle

\setlength{\baselineskip} {15.2pt}

\begin{abstract}
The notion of strict closedness of rings was given by J. Lipman \cite{L} in connection with a conjecture of O. Zariski. The present purpose is to give a practical method of construction of strictly closed rings. It is also shown that the Stanley-Reisner rings of simplicial complexes (resp. $F$-pure rings satisfying the condition $(\operatorname{S}_2)$ of Serre) are strictly closed (resp. weakly Arf) rings.
\end{abstract}

\section{Introduction}\label{intro}

This paper aims at giving a practical method of construction of strictly closed rings. The notion of strict closedness of rings was firstly given by J. Lipman in his famous paper \cite{L} in connection with the Arf property. Let $S/R$ be an extension of commutative rings, and define
$$
R^{*}=\left\{ \alpha \in S \mid \alpha \otimes 1 = 1 \otimes \alpha \text{ in } S \otimes_RS\right\}.
$$
Then $R^*$ forms a subring of $S$ containing $R$, which is called {\it the strict closure} of $R$ in $S$. We say that $R$ is {\it strictly closed} in $S$, if $R=R^{*}$. Since $(R^*)^*=R^*$ in $S$, $R^*$ is strictly closed in $S$ (\cite[Section 4, p.672]{L}), so that $[*]^*$ is actually a closure operation. We simply say that $R$ is {\it strictly closed}, when  $R$ is strictly closed in $\overline{R}$, where $\overline{R}$ denotes the integral closure of $R$ in its total ring $\operatorname{Q}(R)$ of fractions. The reader may consult with \cite{L, CCCEGIM} about properties of strict closures. Here let us explain some of them  for the later use in the present paper.

Let $R$ be a Noetherian semi-local ring and assume that $R$ is a Cohen-Macaulay ring of dimension one, which means that for every $M \in \Max R$ $R_M$ is a Cohen-Macaulay local ring of dimension one. After \cite{L} we say that $R$ is an Arf ring, if the following two conditions are satisfied.
\begin{enumerate}[$(1)$]
\item Every integrally closed ideal $I$ in $R$ that contains a non-zerodivisor has {\it a principal reduction}, i.e. $I^{n+1} = a I^n$ for some $n \ge 0$ and $a \in I$. 
\item Let $x,y,z \in R$ and assume that $x$ is a non-zerodivisor of $R$ and $\frac{y}{x}, \frac{z}{x} \in \overline{R}$. Then $\frac{yz}{x} \in R$.
\end{enumerate}
The conditions are  equivalent to saying that every integrally closed ideal $I$ of $R$ is stable, that is $I^2=aI$ for some $a \in I$, once $I$ contains a non-zerodivisor of $R$. By means of this notation, Lipman extends, to arbitrary Cohen-Macaulay semi-local rings  of dimension one, the result of C. Arf \cite{Arf} about the multiplicity sequences of curve singularities.

The connection between the notion of Arf ring and strictly closed ring is expressed as a conjecture of Zariski. As is explained in \cite{L}, Zariski conjectured that for a Cohen-Macaulay semi-local ring $R$ of dimension one, $R$ is an Arf ring if and only if $R$ is strictly closed. Zariski himself proved the {\it if} part (\cite[Proposition 4.5]{L}), and Lipman showed the converse is also true (\cite[Theorem 4.6]{L}), when $R$ contains a field. Until it was recently proven by \cite[Theorem 4.4]{CCCEGIM} with full generality, it seems that Zariski's conjecture has been open for more than one half of century.

Let us say that a given commutative ring $R$ is weakly Arf, if it satisfies Condition $(2)$ in the above definition of the Arf rings. As Lipman predicted in \cite{L} and as the authors of \cite{CCCEGIM} are now developing a fruitful general theory, there is also a strong connection between the strict closedness and the weak Arf property. In fact, an arbitrary commutative ring $R$ is a weakly Arf ring, once it is strictly closed (\cite[Proof of Proposition 4.5]{L}), and as for Noetherian rings $R$ satisfying the condition $({\rm S}_2)$ of Serre, it is known by \cite[Corollary 4.6]{CCCEGIM} that $R$ is strictly closed if and only if $R$ is weakly Arf and the local ring $R_\fkp$ is Arf for every $\fkp \in \Spec R$ with $\height_R \fkp = 1$, and therefore, provided that $\height_RM \ge 2$ for every $M \in \Max R$, $R$ is strictly closed if and only if $R$ is weakly Arf (\cite[Corollary 4.6]{CCCEGIM}).

The strict closure behaves well with respect to the standard operations, such as localizations, polynomial extensions, and faithfully flat extensions (\cite[Proposition 4.2, Lemma 4.9]{CCCEGIM}, \cite[Proposition 4.3]{L}). It is proved in \cite[Corollary 13.6]{CCCEGIM} that the invariant subrings of strictly closed rings under a finite group action (except the modular case) are still strictly closed. The reader may consult with \cite{L, CCCEGIM} about further study of strict closed rings.

Let us now explain where the motivation for the present researches has come from. 
In general setting, it is not quite easy to determine whether a given  commutative ring $R$ is strictly closed or not. Even for the following simple case $$R=k[X^5,XY^4, Y^5]$$ where $X,Y$ are indeterminates over a field $k$, we meet some puzzled difficulty to compute the strict closure $$R^*=k[X^5, X^9Y^6, X^8Y^7, X^4Y^{11}, XY^4, Y^5]$$ of $R$. In order to overcome the difficulty, we certainly require some practical method of constructing strictly closed rings, which we shall provide in Section 2 of the present paper. As applications, we will explore a few concrete examples, including the above one. The Stanley-Reisner rings of simplicial complexes (resp. $F$-pure rings satisfying the condition $(\operatorname{S}_2)$ of Serre) are typical examples of strictly closed (resp. weakly Arf) rings, which we shall show in Section 3 (resp. Section 4).

Unless otherwise specified, throughout this paper let $R$ denote a commutative ring, and $\overline{R}$ the integral closure of $R$ in its total ring $\operatorname{Q}(R)$ of fractions.


\section{Strict closures of rings}
Let $S/R$ be an extension of commutative rings and we set
$$
R^{*}=\left\{ \alpha \in S \mid \alpha \otimes 1 = 1 \otimes \alpha \text{ in } S \otimes_R S\right\},
$$
which forms a subring of $S$ containing $R$. We say that $R$ is {\it strictly closed} in $S$, if $R=R^{*}$. Since $(R^*)^*=R^*$ in $S$, $R^*$ is strictly closed in $S$ (\cite[Section 4, p.672]{L}). We simply say that $R$ is strictly closed, if $R$ is strictly closed in $\overline{R}$. We refer the reader to \cite{L, CCCEGIM} for the properties of strict closures. Our present purpose is to give a practical method of producing strictly closed rings, exploring a few concrete examples.





Our key is  the following.

\begin{thm}\label{2.1}
Let $R$ be a commutative ring and $T$ an $R$-subalgebra of the total ring $\operatorname{Q}(R)$ of fractions of $R$. Let $V$ be a non-empty subset of $T$ such that $T=R[V]$. If $fg \in R$ for all $f,g \in V$, then $R$ is strictly closed in $T$.
\end{thm}

We divide the proof of Theorem \ref{2.1} into a few steps. Firstly, let $\alpha \in T$ such that $\alpha \otimes 1 = 1 \otimes \alpha$ in $T \otimes_RT$. Then, there exists a finitely generated $R$-submodule $L$ of $T$ which contains both the elements $\alpha$ and $1$ such that $\alpha \otimes 1 = 1 \otimes \alpha$ in $L\otimes_RL$. Since $L$ is finitely generated, we get a finite subset $W$ of $V$, so that $L \subseteq R[W]$. Because
$$L \subseteq R[W] \subseteq R[V] =T \subseteq \operatorname{Q}(R),$$
we then have $\alpha \otimes 1 = 1 \otimes \alpha$ in $R[W]$. Therefore, passing to $W$, without loss of generality we may assume that $V$ is a finite set.

Let $V = \{f_1, f_2, \ldots, f_n\}$. Then $T=R+\sum_{i=1}^nRf_i$. We consider the $R$-linear map
$$
\varepsilon : R^{\oplus (n+1)} \to T
$$
defined by $$\varepsilon \left(\begin{smallmatrix}
a_0\\
a_1\\
\vdots\\
a_n
\end{smallmatrix}\right)=a_0 + a_1f_1+\ldots +a_nf_n$$ for all $\left(\begin{smallmatrix}
a_0\\
a_1\\
\vdots\\
a_n
\end{smallmatrix}\right)\in R^{\oplus (n+1)}.$
Consequently, if $\left(\begin{smallmatrix}
a_0\\
a_1\\
\vdots\\
a_n\\
\end{smallmatrix}\right)\in \Ker \varepsilon$, then $a_0 = -\sum_{i=1}^na_if_i,$
so that $$a_0f_j=- \sum_{i=1}^n(a_if_i)f_j=- \sum_{i=1}^na_i(f_if_j) \in R$$ for all $1 \le j \le n$. Hence, $a_0T \subseteq R$. Therefore, denoting by $R:T$ be the conductor ideal, we get the following.

\begin{lem}\label{2.2}
We have $a_0 \in R:T$ for every $\left(\begin{smallmatrix}
a_0\\
a_1\\
\vdots\\
a_n
\end{smallmatrix}\right)\in \Ker \varepsilon$.
\end{lem}

\begin{proof}[Proof of Theorem $\ref{2.1}$]
Let us choose a non-zero free $R$-module $G$ and a homomorphism $\varphi : G \to R^{\oplus (n+1)}$ so that $\Im \varphi = \Ker \varepsilon$. We consider the exact sequence 
$$
T \otimes_RG \overset{T \otimes_R\varphi}{\to} T\otimes_RR^{\oplus (n+1)} \overset{T\otimes_R\varepsilon}{\to} T\otimes_RT \to 0
$$
of $T$-modules induced from the exact sequence
$$
G \overset{\varphi}{\to} R^{\oplus (n+1)} \overset{\varepsilon}{\to} T \to 0.
$$
Let $\{g_\lambda\}_{\lambda \in \Lambda}$ be a free basis of $G$ and we naturally  identify
$$
T\otimes_RG = T^{\oplus \Lambda}\ \ \text{and}\ \ T \otimes_RR^{\oplus (n+1)} =T^{\oplus (n+1)}.
$$

Now let $\alpha \in R^*$, where $R^*$ denotes the strict closure of $R$ in $T$, and write $$\alpha = a_0+ a_1f_1+ \ldots + a_nf_n$$ with $a_i \in R$. We set $\beta = a_1f_1 + a_2f_2 + \ldots+a_nf_n$. Then, since $\beta \in R^*$, we have $$(-\beta)\otimes 1 +  1\otimes \beta=\sum_{i=1}^na_i\left[(-f_i)\otimes 1 + 1 \otimes f_i\right]=0$$ in $T\otimes_RT =T^{\oplus (n+1)}$. Because 
$$(-f_i)\otimes 1 + 1 \otimes f_i = \varepsilon \left(\begin{smallmatrix}
-f_i\\
0\\
\vdots\\
0\\
1\\
0\\
\vdots\\
0
\end{smallmatrix}\right) 
$$
for each $1 \le i \le n$
(here in the vector of the right hand side the identity $1$ appears at the $i$-th entry), we have 
$$
\sum_{i=1}^na_i\left(\begin{smallmatrix}
-f_i\\
0\\
\vdots\\
0\\
1\\
0\\
\vdots\\
0
\end{smallmatrix}\right) \in \Ker~ (T\otimes_R\varepsilon) = \Im ~(T \otimes_R\varphi). 
$$
Therefore, the sum $\sum_{i=1}^na_i(-f_i)$ belongs to the ideal $I$ of $T$ generated by the entries of the first row of the matrix $\varphi =[g_\lambda]$. Consequently, Lemma \ref{2.2} shows that $$
\sum_{i=1}^na_i(-f_i) \in R:T \subseteq R.
$$
Hence, $\beta \in R$, so that $\alpha = a_0 + \beta  \in R$. Thus, $R$ is strictly closed in $T$.
\end{proof}

\begin{corollary}\label{2.2a}
Let $R$ be a commutative ring and assume that $\overline{R}=R[f]$ for some $f \in \operatorname{Q}(R)$. If $f^2 \in R$, then $R$ is strictly closed, so that $R$ is a weakly Arf ring.
\end{corollary}

\begin{cor}
Let $R$ be a commutative ring and $J =(a_1, a_2, \ldots, a_{n})~(n \ge 3)$ an ideal of $R$ such that $a_1^2 = a_2a_3$. We set $I = (a_2, a_3, \ldots, a_{n})$ and consider the Rees algebras 
$$
\calR = \calR(I) = R[It] \subseteq \calT  = \calR(J) = R[Jt] \subseteq R[t]
$$
of $I$ and $J$, respectively, where $t$ denotes an indeterminate over $R$. Then $\calR$ is strictly closed in $\calT$, provided $I$ contains a non-zerodivisor of $R$.
\end{cor}

\begin{proof}
We have $\operatorname{Q}(\calR)=\operatorname{Q}(\calT)= \operatorname{Q}(R[t])$, because $I$ contains a non-zerodivisor of $R$. Therefore, the assertion follows from Theorem \ref{2.1}, since $\calT = \calR[a_1t]$ and $(a_1t)^2 \in \calR$.
\end{proof}

\begin{corollary}\label{2.3b}
Let $T$ be an integrally closed integral domain and $A$ a subring of $T$. Choose a non-empty subset $V$ of $T$ so that $T=A[V]$ and $0 \not\in V$. We set $$R= A\left[\{fg \mid f,g \in V\}, \{f^3 \mid f \in V\}\right].$$ Then, $T= \overline{R}$, and $R$ is strictly closed, whence $R$ is a weakly Arf ring.
\end{corollary}

\begin{proof}
Since $f^2, f^3 \in R$ for all $f \in V$, $V \subseteq \operatorname{Q}(R)$, whence  $\operatorname{Q}(R)=\operatorname{Q}(T)$, and $T = \overline{R}$. By Theorem \ref{2.1} $R$ is strictly closed in $T$, because $fg \in R$ for all $f,g \in V$.
\end{proof}

We explore more concrete examples. Let us begin with the following.

\begin{Example}
Let $S=k[X,Y]$ be the polynomial ring over a field $k$ and set $$R=k[X^5, XY^4, Y^5]$$
in $S$. Then $$R^*=k[X^5, X^9Y^6, X^{8}Y^7, X^4Y^{11}, XY^4,Y^5].$$
\end{Example}

\begin{proof} Notice that $$\overline{R}=k[X^{5},X^{4}Y ,X^{3}Y^2 ,X^{2}Y^3 ,XY^4, Y^5].$$ We set$$T=k[X^5, X^{13}Y^7, X^9Y^6,  X^4Y^{11}, XY^4,Y^5].$$ Then, because $R^*$ is strictly closed and all the elements $$X^9Y^6= \frac{X^5Y^5{\cdot}X^5Y^5}{XY^4}, \ \ X^{13}Y^7= \frac{X^9Y^6{\cdot}X^9Y^6}{X^5Y^5}, \ \ X^4Y^{11}=\frac{X^2Y^8{\cdot}X^2Y^8}{Y^5}$$ belong to $[R^*]^*$, we get $T \subseteq R^*$, whence $T^*=R^*$ (\cite[Section 4, p.672]{L}).

We now notice that $\overline{R}= T+T{\cdot}X^4Y+T{\cdot}X^3Y^2+T{\cdot}X^2Y^3$. Let $\varepsilon : T^{\oplus 4} \to \overline{R}$ be the $T$-linear map defined by $$\varepsilon \left(\begin{smallmatrix}
a_0\\
a_1\\
a_2\\
a_3
\end{smallmatrix}\right)
=a_0+a_1X^4Y+a_2X^3Y^2+a_3X^2Y^3.$$
Then, it is just a routine work to check that $\Ker \varepsilon $ is generated by the following vectors:

\bigskip
\noindent
{\footnotesize 
$
\left(\begin{smallmatrix}
X^5Y^5\\
-XY^4\\
0\\
0
\end{smallmatrix}\right), 
\left(\begin{smallmatrix}
X^9Y^6\\
-X^5Y^5\\
0\\
0
\end{smallmatrix}\right),
\left(\begin{smallmatrix}
X^{13}Y^7\\
-X^9Y^6\\
0\\
0
\end{smallmatrix}\right),
\left(\begin{smallmatrix}
X^{17}Y^8\\
-X^{13}Y^7\\
0\\
0
\end{smallmatrix}\right),
$\\
$
\left(\begin{smallmatrix}
X^9Y^6\\
0\\
-X^6Y^4\\
0\\
\end{smallmatrix}\right),
\left(\begin{smallmatrix}
X^{13}Y^7\\
0\\
-X^{10}Y^5\\
0
\end{smallmatrix}\right),
\left(\begin{smallmatrix}
X^{12}Y^8\\
0\\
-X^9Y^6\\
0
\end{smallmatrix}\right),
\left(\begin{smallmatrix}
X^{16}Y^9\\
0\\
-X^{13}Y^7\\
0
\end{smallmatrix}\right),
\left(\begin{smallmatrix}
X^4Y^{11}\\
0\\
-XY^9\\
0
\end{smallmatrix}\right),
\left(\begin{smallmatrix}
X^3Y^{12}\\
0\\
-Y^{10}\\
0
\end{smallmatrix}\right),
\left(\begin{smallmatrix}
X^7Y^{13}\\
0\\
-X^4Y^{11}\\
0
\end{smallmatrix}\right),
\left(\begin{smallmatrix}
X^5Y^{10}\\
0\\
-X^2Y^{8}\\
0
\end{smallmatrix}\right),
$\\
$
\left(\begin{smallmatrix}
X^{13}Y^{7}\\
0\\
0\\
-X^{11}Y^4\\
\end{smallmatrix}\right),
\left(\begin{smallmatrix}
X^2Y^{8}\\
0\\
0\\
-Y^5\\
\end{smallmatrix}\right),
\left(\begin{smallmatrix}
X^{11}Y^{9}\\
0\\
0\\
-X^9Y^6\\
\end{smallmatrix}\right),
\left(\begin{smallmatrix}
X^{15}Y^{10}\\
0\\
0\\
-X^{13}Y^7\\
\end{smallmatrix}\right),
\left(\begin{smallmatrix}
X^3Y^{12}\\
0\\
0\\
-XY^9\\
\end{smallmatrix}\right),
\left(\begin{smallmatrix}
X^6Y^{14}\\
0\\
0\\
-X^4Y^{11}\\
\end{smallmatrix}\right),
\left(\begin{smallmatrix}
X^4Y^{11}\\
0\\
0\\
-X^2Y^8\\
\end{smallmatrix}\right),
$\\
$
\left(\begin{smallmatrix}
0\\
Y^{5}\\
-XY^4\\
0
\end{smallmatrix}\right),
\left(\begin{smallmatrix}
0\\
X^9Y^{6}\\
-X^{10}Y^5\\
0
\end{smallmatrix}\right),
\left(\begin{smallmatrix}
0\\
X^{13}Y^{7}\\
-X^{14}Y^6\\
0
\end{smallmatrix}\right),
\left(\begin{smallmatrix}
0\\
X^{12}Y^{8}\\
-X^{13}Y^7\\
0
\end{smallmatrix}\right),
\left(\begin{smallmatrix}
0\\
X^4Y^{11}\\
-X^5Y^{10}\\
0
\end{smallmatrix}\right),
\left(\begin{smallmatrix}
0\\
X^3Y^{12}\\
-X^4Y^{11}\\
0
\end{smallmatrix}\right),
$\\
$
\left(\begin{smallmatrix}
0\\
X^9Y^{6}\\
0\\
-X^{11}Y^{4}\\
\end{smallmatrix}\right),
\left(\begin{smallmatrix}
0\\
X^{13}Y^{7}\\
0\\
-X^{5}Y^{5}\\
\end{smallmatrix}\right),
\left(\begin{smallmatrix}
0\\
X^7Y^{8}\\
0\\
-X^{9}Y^{6}\\
\end{smallmatrix}\right),
\left(\begin{smallmatrix}
0\\
Y^{10}\\
0\\
-X^{2}Y^{8}\\
\end{smallmatrix}\right),
\left(\begin{smallmatrix}
0\\
X^4Y^{11}\\
0\\
-X^{6}Y^{9}\\
\end{smallmatrix}\right),
\left(\begin{smallmatrix}
0\\
X^3Y^{12}\\
0\\
-X^{5}Y^{10}\\
\end{smallmatrix}\right),
\left(\begin{smallmatrix}
0\\
X^{11}Y^{9}\\
0\\
-X^{13}Y^{7}\\
\end{smallmatrix}\right),
\left(\begin{smallmatrix}
0\\
X^{2}Y^{13}\\
0\\
-X^{4}Y^{11}\\
\end{smallmatrix}\right),
$\\
$
\left(\begin{smallmatrix}
0\\
0\\
Y^{5}\\
-XY^{4}\\
\end{smallmatrix}\right),
\left(\begin{smallmatrix}
0\\
0\\
X^{9}Y^{6}\\
-X^{10}Y^{5}\\
\end{smallmatrix}\right),
\left(\begin{smallmatrix}
0\\
0\\
X^{13}Y^{7}\\
-X^{14}Y^{6}\\
\end{smallmatrix}\right),
\left(\begin{smallmatrix}
0\\
0\\
X^{12}Y^{8}\\
-X^{13}Y^{7}\\
\end{smallmatrix}\right),
\left(\begin{smallmatrix}
0\\
0\\
X^{4}Y^{11}\\
-X^{5}Y^{10}\\
\end{smallmatrix}\right),
\left(\begin{smallmatrix}
0\\
0\\
X^{3}Y^{12}\\
-X^{4}Y^{11}\\
\end{smallmatrix}\right).
$}

\bigskip

\noindent
Therefore, Proof of Theorem \ref{2.1} guarantees that $T^* \subseteq T + J$, where $J$ denotes the ideal of $\overline{R}$ generated by the entries of the first rows in the vectors above, so that we have $$T^* \subseteq T+J=T + kX^8Y^7$$ since $$T+J = T + X^5Y^5{\cdot}\overline{R}.$$ On the other hand,
because
{\footnotesize
$$
\begin{pmatrix}
X^8Y^7\\
0\\
0\\
-X^6Y^4
\end{pmatrix}
= X^3Y^2
\begin{pmatrix}
X^5Y^5\\
- XY^4\\
0\\
0
\end{pmatrix}
+ X^4Y
\begin{pmatrix}
0\\
Y^5\\
- XY^4\\
0
\end{pmatrix}+ X^5
\begin{pmatrix}
0\\
0\\
Y^5\\
-XY^4
\end{pmatrix},
$$}

\noindent
we see $X^8Y^7 \otimes 1 = 1 \otimes X^8Y^7$ in $\overline{R} \otimes_T\overline{R}$, that is $X^8Y^7 \in T^*$, which guarantees $T^*=T+kX^8Y^7$. Consequently $$R^*= k[X^5,X^9Y^6,X^8Y^7,X^4Y^{11},XY^{4},Y^5], $$
since $T^* = R^*$.
\end{proof}

\begin{ex}
Let $(R, \m)$ be a two-dimensional regular local ring with $\m =(x,y)$. Let $I = (x^3, xy^4, y^5)$ and $\calR= R[It]$, where $t$ denotes an indeterminate. Then $\calR$ is a strictly closed ring with $\overline{\calR}=R[Jt]$, where $J = (x^3, x^2y^2, xy^4, y^5)$. Therefore $\overline{\calR}\ne \calR$. 
\end{ex}

\begin{proof}
Let $J = (x^3, x^2y^2, xy^4, y^5)$ and  set $T = R[Jt]$. Then, since $J$ is an integrally closed ideal of $R$, we get $T = \overline{\calR}$, while $T = \calR[x^2y^2t]$ and $(x^2y^2t)^2 \in \calR$. Hence, by Corollary \ref{2.2a} $\calR$ is strictly closed.
\end{proof}

\begin{Example}
Let $S =k[X,Y]$ be the polynomial ring over a field $k$. Let $n \ge 3$ be an integer and set $$R=k[X^{n-i}Y^i \mid 0 \le i \le n, \ i\ne 1]$$ in $S$. Then $R$ is a strictly closed Cohen-Macaulay ring with $\dim R=2$.
\end{Example}

\begin{proof}Let $T= k[X^{n-i}Y^i \mid 0 \le i \le n]$. Then, $T = \overline{R}$, and $T=R[X^{n-1}Y]$, so that $R$ is strictly closed by Corollary \ref{2.2a}, because $(X^{n-1}Y)^2 \in R$. Since $X^n,Y^n$ forms a regular sequence in $R$, $R$ is a Cohen-Macaulay ring of dimension $2$. 
\end{proof}




\begin{Example}\label{2.7b}
Let $S=k[X_1, X_2, \ldots, X_n]$  $(n \ge 1)$  be the polynomial ring over a field $k$ and set $$R = k[\{X_iX_j \mid 1 \le i \le j \le n\}, \{X_i^3 \mid 1 \le i \le n\}]$$ in $S$. Then $R$ is a strictly closed ring such that $R \ne \overline{R}$.
\end{Example}

The condition that $R= A\left[\{fg \mid f,g \in V\}, \{f^3 \mid f \in V\}\right]$ in Corollary \ref{2.3b} is in some sense the best possible as the following example shows. Notice that $R_2$ is not a weakly Arf ring, because $\frac{X^3}{X^2}, \frac{X^2Y}{X^2} \in S$ but $\frac{X^3{\cdot}X^2Y}{X^2}=X^3Y \not\in R_2$.

\begin{remark}
Let $S = k[X,Y]$ be the polynomial ring over a field $k$. We set $$R_1= k[X^2, XY, Y^2, X^3,Y^3]\ \ \  \text{and}\ \ \ R_2 = k[X^2, X^2Y, Y^2, X^3, Y^3]$$
in $S$. Then $S = \overline{R_1} = \overline{R_2}$. The ring $R_i$ is strictly closed, while $R_2$ is not weakly Arf, so that  $R_2$ is not strictly closed.
\end{remark}




\section{Stanley-Reisner rings}
The purpose of this section is to show that Stanley-Reisner algebras are strictly closed rings. Let $R$ be a Noetherian reduced ring. We write $\Min R=\{\p_1, \p_2, \ldots, \p_{\ell}\}$ where $\ell = \sharp \Min R \ge 2$. We assume that the following specific condition is satisfied:
\begin{center}
$(*)$ $R/{\p_i}$ is integrally closed for every $1 \le i \le \ell$. 
\end{center}
We then have the embedding
$$
0 \to R \overset{\varphi}{\to} R/{\p_1} \oplus R/{\p_2} \oplus \cdots \oplus R/{\p_{\ell}} = \overline{R},
$$
where $\varphi(\alpha) = (\overline{\alpha}, \overline{\alpha}, \ldots, \overline{\alpha})$ for each $\alpha \in R$ (here we denote by $\overline{\alpha}$ the image of $\alpha$ in each $R/\p_i$).  Let $e_i = (0, \ldots, 0, \overset{\overset{i}{\vee}}{1}, 0, \ldots, 0) \in \overline{R}$ for each $1 \le i \le \ell$. Hence $$\overline{R} = \sum_{i=1}^{\ell} Re_i = \bigoplus_{i=1}^{\ell} Re_i$$
and 
$$
\overline{R}\otimes_R \overline{R} = \sum_{1 \le i, j \le \ell} R(e_i \otimes e_j) = \bigoplus_{1 \le i, j \le \ell}  Re_i \otimes_R Re_j. 
$$ 
Let $\alpha \in \overline{R}$ and write $\alpha = (\overline{\alpha_1}, \overline{\alpha_2}, \ldots, \overline{\alpha_{\ell}})$ with $\alpha_i \in R$. We then have 
$$
\alpha \otimes 1 = \sum_{1 \le i, j \le \ell}\alpha_i(e_i\otimes e_j) \ \ \text{and} \ \   1 \otimes \alpha = \sum_{1 \le i, j \le \ell}\alpha_j(e_i\otimes e_j)
$$
in  $\overline{R}\otimes_R \overline{R}$. Therefore, 
$\alpha \otimes 1 = 1 \otimes \alpha$ if and only if $$\alpha_i(e_i\otimes e_j) = \alpha_j(e_i\otimes e_j)$$ for all $1 \le i, j \le \ell$. Since $$R(e_i \otimes e_j)=Re_i \otimes_R Re_j \cong R/{\p_i} \otimes_R R/{\p_j} \cong R/{[\p_i + \p_{j}]},$$ the latter condition is equivalent to saying that $$\alpha_i - \alpha_j \in \p_i + \p_j$$ for all $1 \le i, j \le \ell$.

With this notation we have the following. 

\begin{prop}\label{2.3}
Suppose that $\ell = 2$. Then $R$ is strictly closed, so that it is a weakly Arf ring. 
\end{prop}

\begin{proof}
Let $\alpha \in \overline{R}$ and assume that $\alpha \otimes 1 = 1 \otimes \alpha$ in $\overline{R}\otimes_R \overline{R}$. We write $\alpha = (\overline{\alpha_1}, \overline{\alpha_2})$ with $\alpha_i \in R$. Then, since $\alpha_1 - \alpha_2 \in \p_1 + \p_2$, we get $\alpha_1 - \alpha_2 = \beta_1 + \beta_2$ for some $\beta_1 \in \p_1$ and $\beta_2 \in \p_2$. Because
\begin{eqnarray*}
\alpha = (\overline{\alpha_1}, \overline{\alpha_2}) = (\overline{\alpha_2 + \beta_1 + \beta_2}, \overline{\alpha_2}) =  (\overline{\alpha_2 + \beta_2}, \overline{\alpha_2}) = (\overline{\alpha_2 + \beta_2}, \overline{\alpha_2 + \beta_2}),
\end{eqnarray*}
we have $\alpha \in R$, whence $R = R^*$ in $\overline{R}$.
\end{proof}

\begin{corollary}
Let $S$ be a regular local ring and let $a_1, a_2, \ldots, a_m, b_1, b_2, \ldots, b_n$ be a regular system of parameters of $S$ $($hence $m + n = \dim S$$)$. Then $$R = S/[(a_1, a_2, \ldots, a_m)\cap(b_1, b_2, \ldots, b_n)]$$ is a strictly closed ring, so that it is weakly Arf.
\end{corollary}

We  now consider the case where $R$ is the Stanley-Reisner ring of a simplicial complex. 
Let $V=\{ 1, 2, \ldots, n\}$ be a vertex set and $\Delta$ a simplicial complex on $V$. We assume $\Delta \neq \emptyset$ and denote by $\calF(\Delta)$ the set of facets of $\Delta$. We write $\calF(\Delta) =\{F_1, F_2, \ldots, F_{\ell}\}$ where  $\ell=\sharp \calF(\Delta)$. Let $S=k[X_1, X_2, \ldots, X_n]$ be the polynomial ring over a field $k$ and set $$R =k[\Delta]= S/{I_{\Delta}},$$where $I_{\Delta}= \bigcap_{i=1}^{\ell} P_i$ and $P_i = (X_{\alpha} \mid \alpha \not\in F_i)$ for each $1 \le i \le \ell$.

We then have the following.

\begin{thm}\label{main2}
The Stanley-Reisner ring $R=k[\Delta]$ of $\Delta$ is strictly closed, whence $R$ is a weakly Arf ring. 
\end{thm}


\begin{proof}
We consider $S$ to be naturally a $\Bbb Z^n$-graded ring. Then, since both of the rings $R$ and $\overline{R}$ are $\Bbb Z^n$-graded, we get the commutative diagram
$$
\xymatrix{
S \ar[r]\ar[d] & S/{P_1} \oplus S/{P_2} \oplus \cdots \oplus S/{P_{\ell}} \ar[d]^{\cong}  \\
R=S/I_{\Delta} \ar[r] & R/{\p_1} \oplus R/{\p_2} \oplus \cdots \oplus  R/{\p_{\ell}} = \overline{R}
}
$$
of $\Bbb Z^n$-graded rings, where $\p_i = P_i/I_{\Delta}$ for each $1 \le i \le \ell$ and the vertical maps are canonical. Notice that $R^* = \Ker \varphi$ is a $\Bbb Z^n$-graded subring of $\overline{R}$, since the $R$-linear map
$$
\varphi : \overline{R} \to \overline{R} \otimes_R \overline{R}, \ \alpha \mapsto \alpha \otimes 1 - 1 \otimes \alpha
$$
is a homomorphism of graded $R$-modules.

We now assume that $R^* \not\subseteq R$ and choose a homogeneous element $\alpha \in R^* \setminus R$. We set $h = \deg \alpha$ and write $h = (a_1, a_2, \ldots, a_n)$ where $a_i \ge 0$ for each $1 \le i \le n$. Then, since $\overline{R} = R/{\p_1} \oplus R/{\p_2} \oplus \cdots \oplus  R/{\p_{\ell}}$, we get  
$$
\alpha = (\overline{\alpha_1}, \overline{\alpha_2}, \ldots, \overline{\alpha_{\ell}})
$$
with $\alpha_i \in R_{h}$. Let us choose elements $f_i \in S_h$ so  that $\alpha_i = \overline{f_i}$ in $R$. Then, since $\alpha \in R^*$, we get $\alpha_i - \alpha_j \in \p_i + \p_j$ for all $1 \le i, j \le \ell$, as we have shown above. Therefore
\begin{center}
$f_i - f_j \in P_i + P_j$
\end{center}
 for all $1 \le i, j \le \ell$. Let us write $f_i = c_i X^h$ with $c_i \in k$, where $X^h = X_1^{a_1} X_2^{a_2} \cdots X_n^{a_n}$. 
We now set 
$$
\Gamma = \{ \ i \mid 1 \le i \le \ell ~\text{and}~f_i \not\in P_i\}.
$$ 
Then $\Gamma \ne \emptyset$, and for all $i, j \in \Gamma$  we have
$$
(c_i - c_j) X^h \in P_i + P_j. 
$$
Therefore, if $c_i \ne c_j$, then $X^h \in P_i + P_j$, so that either $X^h \in P_i$ or $X^h\in P_j$ (because each $P_i$ is a graded ideal of $S$ generated by monomials in $X_i$'s), which implies that either $f_i \in P_i$ or $f_j \in P_j$. This is absurd, whence $c_i = c_j$ for all $i, j \in \Gamma$. We set $c = c_i$ ($i \in \Gamma$); hence $c \ne 0$.

If $\sharp \Gamma = \ell$, we then have  $f_1 = f_2 = \cdots = f_{\ell}$, so that $\alpha \in R$. Therefore, $\sharp\Gamma < \ell$, and  we may choose at least one index $1 \le i \le \ell$ so that $i \not\in \Gamma$. For each $j \in \Gamma$, we then have $$(c_i - c) X^h \in P_i + P_j,$$ while $c_i = 0$ or $X^h \in P_i$, since  $c_i X^h \in P_i$. If $X^h \not\in P_i$, then $c_i = 0$, so that $$cX^h \in P_i + P_j,$$ which shows $X^h \in P_j$. This is impossible, as $j \in \Gamma$. Therefore, $X^h \in P_i$ for each $i \not\in \Gamma$, and so, setting $f = c X^h$ and $\beta = \overline{f}$  in $R$, we get 
\begin{center}
$\overline{\beta} = \overline{\alpha_i}$ \ in $R/\p_i$
\end{center}
for every $1 \le i \le \ell$. Hence, $\alpha \in R$, which gives the final contradiction. Therefore, $R=k[\Delta]$ is strictly closed in $\overline{R}$. 
\end{proof}

One of the simplest example of Stanley-Reisner rings is the following. When the characteristic of the field $k$ is positive, this ring is so called $F$-pure, and the weak Arf property of $R$ can be deduced in a different way, which we shall discuss in the next section.

\begin{Example}
Let $S = k[X_1, X_2, \ldots, X_n]$ $(n \ge 3)$ be the polynomial ring over a field $k$ and set $$R= S/(X_1 \cdots \overset{\vee}{X_i}\cdots X_n \mid 1 \le i \le n).$$ Then, $R$ is a strictly closed Cohen-Macaulay ring of dimension one, so that it is weakly Arf.
\end{Example}


\section{$F$-pure rings}
The purpose of this section is to explore certain $F$-pure rings are always weakly Arf.  Before going ahead, we need some preliminaries.

\begin{prop}[{cf. \cite[Corollary 3.14]{GMP}}]\label{prop2}
Let $(R, \m)$ be an arbitrary local ring $($not necessarily Noetherian, and not necessarily of positive characteristic$)$. If $\m \overline{R} \subseteq R$, then $R$ is a weakly Arf ring. 
\end{prop}

\begin{proof}
Let $x, y, z \in R$. Assume that $x$ is a non-zerodivisor of $R$ and that $\frac{y}{x}, \frac{z}{x} \in \overline{R}$. We have to show $\frac{yz}{x} \in R$. To do this, we may assume $x \in \m$. We then have
$$
\frac{yz}{x} = x \cdot \frac{y}{x} \cdot\frac{z}{x} \in \m \overline{R} \subseteq R.
$$
Hence, $R$ is a weakly Arf ring.
\end{proof}

\begin{remark}
Let $R$ be a Cohen-Macaulay local ring with $\dim R=1$. If $\m \overline{R} \subseteq R$, then $R$ is an analytically unramified local ring and $R$ is an almost Gorenstein ring (\cite[Corollary 3.12]{GMP}).
\end{remark}

\begin{Example}
Let $S = k[[X_1, X_2, \ldots, X_d]]$ $(d > 0)$ be the formal power series ring over a field $k$ and let $\n =(X_1, X_2, \ldots, X_d)$, the maximal ideal of $S$. Let $J$ be an arbitrary $\n$-primary ideal of $S$ and set $R = k + J.$
Then, $R$ is a Noetherian local ring with $\m =J$ the maximal ideal, so that $R$ is a weakly Arf ring, because $S = \overline{R}$ and $\m S =J \subseteq R$.
\end{Example}

We are now back to the main topic of this section. Let $R$ be a commutative ring, containing a field of positive characteristic $p >0$. We denote $R$ by $S$ when we regard $R$ as an $R$-algebra via the Frobenius map $$F : R \to R,\  a \mapsto a^p.$$ Notice that $R$ is a reduced ring if and only if $F : R \to S$ is injective. We say that the ring $R$ is {\it $F$-pure}, if for each $R$-module $M$ the homomorphism $$M \to S \otimes_R M, \  m \mapsto 1 \otimes m$$ is injective. The Frobenius map naturally induces the homomorphism $$f :\operatorname{Q}(R)/R \to \operatorname{Q}(R)/R, \ \overline{a} \mapsto \overline{a^p}$$ of $R$-modules, where $\overline{a}$ (resp. $\overline{a^p}$) denotes, for each $a \in \operatorname{Q}(R)$, the image of $a$ in $\operatorname{Q}(R)/R$ (resp. the image of $a^p$ in $\operatorname{Q}(R)/R$). Notice that the homomorphism $f :\operatorname{Q}(R)/R \to \operatorname{Q}(R)/R$ is injective, once $R$ is an $F$-pure ring.

Remember that a Noetherian ring $R$ is said to satisfy the condition $({\rm S}_2)$ of Serre, if $$\depth R_\fkp \ge \inf \{2, \dim R_\fkp \}$$ for every $\fkp \in \Spec R$. We now come to the main result of this section.

\begin{thm}\label{main}
Every Noetherian F-pure ring $R$ satisfying $({\rm S}_2)$ is a weakly Arf ring.
\end{thm}

\begin{proof}
Let $\fkp \in \Spec R$ with  $\depth R_\fkp \le 1$. By \cite[Theorem 2.6]{CCCEGIM} it is enough to show that $R_\fkp$ is a weakly Arf ring. To do this, we may assume $\dim R_\fkp =1$, since $R$ satisfies $({\rm S}_2)$. Hence, $R_\fkp$ is a one-dimensional Cohen-Macaulay local ring. Because $F$-purity is preserved under localization and $\operatorname{Q}(R_\fkp) = [\operatorname{Q}(R)]_\fkp$, without loss of generality we may also assume that $(R, \m)$ is a Cohen-Macaulay local ring with $\dim R=1$. 
Therefore, since the $\m$-adic completion $\widehat{R}$ of $R$ remains $F$-pure (\cite[Lemma 3.26]{Hashimoto}), $\widehat{R}$ is a reduced ring, so that the normalization $\overline{R}$ of $R$ is a module-finite extension of $R$, whence $\ell_R(\overline{R}/R) < \infty$, that is $\m^{\ell} {\cdot}(\overline{R}/R) =(0)$ for some $\ell \gg 0$. Hence, $f^{\ell}(\m {\cdot}(\overline{R}/R)) =(0)$ for all $\ell \gg 0$. Therefore, the injectivity of the homomorphism $f :\operatorname{Q}(R)/R \to \operatorname{Q}(R)/R$ guarantees that $\m {\cdot}(\overline{R}/R)=(0)$. Thus, $R$ is a weakly Arf ring by Proposition \ref{prop2}.  
\end{proof}






Closing this paper, let us  note one consequence, which follows from Theorem \ref{main} and \cite[Corollary 4.6]{CCCEGIM}.


\begin{cor}
Let $R$ be a Noetherian ring of characteristic $p>0$ such that $R$ is $F$-pure, satisfying $({\rm S}_2)$. If one of the following conditions
\begin{enumerate}[$(1)$]
\item $R$ contains an infinite field.
\item $\height_R M \ge 2$ for every $M \in \Max R$.
\end{enumerate}
is satisfied, then $R$ is strictly closed and $R_\fkp$ is an Arf ring for every $\fkp \in \Spec R$ with $\height_R \fkp =1$. 
\end{cor}


\vspace{0.5em}

\begin{ac}
The authors would like to thank Rankeya Datta for valuable comments. 
\end{ac}


\end{document}